\documentclass{amsart}
\usepackage{latexsym}
\usepackage{amsmath,amsfonts,epsfig, graphics, graphicx, amsthm, amssymb}
\input{xy}
\xyoption{all}

\newtheorem{theorem}{Theorem}[section]
\newtheorem{lemma}[theorem]{Lemma}

\newtheorem*{varthm1}{Main Theorem}
\newtheorem*{Ko}{Korselt's Criterion}

\newtheorem{conjtd}{Conjecture}

\newtheorem{ques}{Question}

\begin{document}

\title{Carmichael numbers and least common multiples of $p-1$}
\author{Thomas Wright}
\address{429 N. Church St.\\Spartanburg, SC 29302\\USA}
\maketitle

\begin{abstract}
For a Carmichael number $n$ with prime factors $p_1,\cdots,p_m$, define
$$K=GCD[p_1-1,\cdots,p_m-1],$$
and let $C_\nu(X)$ denote the number of Carmichael numbers up to $X$ such that $K=\nu$.   Assuming a strong conjecture on the first prime in an arithmetic progression, we prove that for any even natural number $\nu$,
$$C_\nu(X)\geq X^{1-(2+o(1))\frac{\log\log\log \log X}{\log\log\log X}}.$$
This is a departure from standard constructions of Carmichael numbers, which generally require $K$ to grow along with $n$.
\end{abstract}

\section{Introduction}

A Carmichael number is a composite integer $n$ such that
$$a^n\equiv a\pmod n$$
for every integer $a$.

While the first Carmichael numbers were discovered over a century ago \cite{Ca}, \cite{Si}, a proof that the set of Carmichael numbers is infinite appeared more recently in 1994 \cite{AGP}.  In that proof, the authors raised a number of further questions, one of which is the following:

\begin{ques}
For any prime $P$, are there infinitely many Carmichael numbers $n$ for which $P|n$?
\end{ques}

Unlike many of the other problems raised in that paper, which have been either resolved completely \cite{La}, partially \cite{WrE}, or at least conditionally \cite{Ch}, \cite{WrV}, this one has seen little progress.


The difficulty here is that the construction in \cite{AGP} requires that all of the primes $p|n$ be such that the $p-1$'s share a large common factor $k$.   Importantly, $k$ must increase as $n$ grows, which means that this method does not allow us to find infinitely many $n$ divisible by a fixed prime $P$.  In fact, if we define
$$K=GCD[p_1-1,\cdots,p_m-1],$$
even the simpler question of finding (unconditionally) infinitely many $n$ for which $K$ is bounded by some constant appears out of reach with the \cite{AGP} construction.

More specifically, let $\lambda(n)$ denote as usual the Carmichael lambda function, defined to be the smallest integer such that
$$a^{\lambda(n)}\equiv 1\pmod n$$
for any $a$ relatively prime to $n$.   The standard construction for Carmichael numbers first creates an $L$ for which $\lambda(L)$ is much smaller than $L$ itself.  From here, one looks for a $k$ such that the set
$$\mathcal P_k=\{p\mbox{ prime}:p=dk+1,d|L,(k,L)=1\}$$
is large.  If there are enough such primes for a given choice of $k$, one can use combinatorial results to find a subset of the primes in $\mathcal P_k$ that multiply to a Carmichael number $n$.  Unfortunately, finding primes of the form $dk+1$ requires results about primes in arithmetic progressions, and these results do not apply unless $p>d^\frac{12}{5}$ \cite{Hu}; in other words, the construction requires that $k>p^\frac{7}{12}$. This required commonality between the prime factors is a clear obstruction to the discovery of Carmichael numbers that are multiples of a fixed prime factor.  After all, if $k\geq P$ for a given $P$ then it is impossible for $P$ to be an element of $\mathcal P_k$, and hence $P$ cannot be a factor of our constructed Carmichael number.

In fact, even if one were to construct these sets $\mathcal P_k$ by assuming the heuristically best possible conjectures about primes in arithmetic progressions (i.e. Conjecture 1 below), one would still require $k\gg \log^2 p$, which, while being an improvement, still goes to infinity as $n$ does the same.



As such, it would seem that an important first step toward a resolution of Question 1 would be to show that $K$ need not go to infinity as $n$ grows large.   In this paper, we find that this can indeed be shown under the assumption of a very strong conjecture on the first prime in an arithmetic progression.  A version of this conjecture was first formulated by Heath-Brown in 1978:

\begin{conjtd}
There exists an $A\geq 2$ such that if $(b,l)=1$ then there exists a prime $p\equiv b\pmod l$ with $$p\ll l\left(\log l\right)^A.$$
\end{conjtd}
So as to avoid $\gg$ notation, we will say that there exists an $A$ such that, for $l$ sufficiently large, there exists a prime $p\equiv b\pmod l$ with
\begin{gather}\label{conjineq}p<l\left(\log  l\right)^A.\end{gather}

This is a conjecture that has been frequently invoked in Carmichael-related papers - see e.g. \cite{BP}, \cite{EPT}, \cite{WrD}, \cite{WrV}.  The full version of the conjecture as stated by Heath-Brown claims that this bound should for every $A\geq 2$.  It is not expected that this conjecture should hold for $A<2$; indeed, Granville and Pomerance have conjectured that the first prime $p\equiv b\pmod l$ should be $\gg \phi(l) (\log l)^2$ for infinitely many choices of $l$ (see \cite{GPprime}, page 2).  In our paper, however, we only require that some such $A$ exists.


Define $C(X)$ to be the number of Carmichael numbers up to $X$, and let $C_\nu(X)$ denote the number of Carmichael numbers up to $X$ for which $K=\nu$.  In this paper, we prove the following:
\begin{varthm1}
Assume Conjecture 1 holds.  Then for any even $\nu$,
$$C_\nu(X)\geq X^{1-(2+o(1))\frac{\log\log \log \log X}{\log \log\log X}}.$$
\end{varthm1}
This is the same lower bound found in \cite{WrD} for the original quantity $C(X)$, and it is close to best possible.  Pomerance \cite{Po} proved that
$$C(X)\leq X^{1-\frac{\log\log \log X}{2\log\log X}}$$
for sufficiently large $X$, and he subsequently conjectured that
$$C(X)\gg X^{1-\frac{\log\log \log X}{\log\log X}}.$$
This would suggest that, while modern construction methods for Carmichael numbers require ever-increasing $K$, the density of Carmichael numbers with bounded $K$ should be relatively close to the number of Carmichael numbers themselves.
\section{Construction Methods}

We begin by stating the well-known necessary and sufficient condition for Carmichael numbers, which Korselt discovered in 1899 \cite{Ko}:


\begin{Ko} \textit{A positive composite integer $n$ is a Carmichael number if and only if $n$ is squarefree and $p-1|n-1$.}\end{Ko}

Nearly every modern effort involving Carmichael numbers follows the framework of \cite{AGP}, which depends heavily upon this criterion; we describe that framework here.  Let $P(y)$ denote the largest prime factor of $y$, and let $\lambda$ denote the Carmichael lambda function.  First, the authors of that paper find a large set of primes $\mathcal Q$ such that for any $q\in \mathcal Q$, $P(q-1)<q^{1-E}$ for some $0<E<1$.  The primes in $\mathcal Q$ are then multiplied together to form
$$L=\prod_{q\in \mathcal Q}q.$$
Because the $q-1$ are smooth relative to $q$, it can be shown that $\lambda(L)$ is small relative to $L$.

Next, the authors define
$$\mathcal P_k=\{p:p=dk+1:d|L,d\leq x^B,(L,k)=1\}$$
for a constant $B<1$.

Using results about primes in arithmetic progressions, one can show that there exists a $k_0\geq x^{1-B}$ such that $\mathcal P_{k_0}$ is large if $B<\frac{5}{12}$.  Using a combinatorial theorem of van Emde Boas and Kruyswijk \cite{EK} and Meshulam \cite{Me}, it can then be shown that there are many subsets $\{p_1,\cdots,p_m\}\subset \mathcal P_{k_0}$ such that
$$n=p_1\cdots p_m\equiv 1\pmod L.$$
Clearly, $n$ is also 1 mod $k_0$, since $n$ is the product of primes that are 1 mod $k_0$.  So for any $p|n$,
$$p-1=dk_0|Lk_0|n-1.$$
Hence, $n$ is a Carmichael number.

Here, we alter the framework in a way that is somewhat similar to \cite{WrD} and \cite{WrF}.  One of the key ideas in those two papers was to change the way we construct $\mathcal Q$ so as to make $\lambda(L)$ even smaller relative to $L$.  In particular, the method used to construct our primes $p$ can also be used to construct our primes $q$.  Let
$$J=\prod_{\substack{\frac z2\leq r\leq z, \\ r\mbox{ prime}}}r,$$
and define
$$\mathcal R_j=\{q\mbox{ prime}:q=gj+1,g|J,\omega(g)=\lfloor \log z\rfloor\}.$$
Just as before, we can find a $j_0$ for which $\mathcal R_{j_0}$ is relatively large.  Here, the primes $q\in \mathcal R_{j_0}$ are such that $q-1|Jj_0$.  Letting $\mathcal Q=\mathcal R_{j_0}$ for some set $\mathcal R_{j_0}$ with many primes, we define $L$ as before and find that $\lambda(L)|Jj_0$ as well.  Since this $\lambda(L)$ is very small relative to $L$, we can use much smaller sets of primes $\mathcal P_k$ to find a subset whose product is 1 modulo $L$.

The major change that we make here is that we create two different (and disjoint) sets $\mathcal Q_1$ and $\mathcal Q_2$.  We then create an analogous $L_1$ and $L_2$ and prime sets $\mathcal P_{k_1}$ and $\mathcal P_{k_2}$, constructed in such a way that $p_1=d_1k_1\nu+1$ and $(p_1-1,L_2k_2)=1$ for $p_1\in \mathcal P_{k_1}$ and vice-versa for $p_2\in \mathcal P_{k_2}$.  Since the $k_i$ are small (as a result of both the construction and the conjecture), it is possible to find sets of primes in $\mathcal P_{k_1}$ that multiply to 1 mod $k_2L_1L_2$ and sets primes in $\mathcal P_{k_2}$ that multiply to 1 mod $k_1L_1L_2$.  From the set $\mathcal P_{k_1}$, then, we create a product $n_1$ comprised of primes in this set such that $n_1\equiv 1\pmod{L_1L_2k_1k_2\nu}$; we do the same to find an $n_2$ from $\mathcal P_{k_2}$ such that $n_2\equiv 1\pmod{L_1L_2k_1k_2\nu}$.  Letting $n=n_1n_2$, we find that $n$ is a Carmichael number with $K=\nu$.

Importantly, we require Conjecture 1 in order to guarantee that $k_1$ and $k_2$ are small.  If, say, $k_2$ were of size $p^{\frac{7}{12}}$ as in \cite{AGP}, or even if $k_2$ were of size $p^\epsilon$ for some small constant $\epsilon$, we would not be able to find enough primes in $\mathcal P_{k_1}$ to guarantee that some subset of them would multiply to 1 modulo $k_2$ (or primes in $\mathcal P_{k_2}$ that multiply to 1 modulo $k_1$).  One could actually weaken the conjecture somewhat and still prove this result - letting $A=\log\log z$ would still allow the result to be proven - however, we use the requirement that $A$ be a constant to simplify the exposition.

We also note that in most cases below (e.g. lower bounds for $R_{j}$ and $\mathcal P_{k_i}$ and upper bounds for $\lambda(L)$ and $L$), the bounds here are not close to sharp and can certainly be improved.  However, such improvements would have no effect on the main term of the Main Theorem; indeed, sharpening these bounds to best possible would only affect the $o(1)$-term.  Hence, we content ourselves with the loose bounds below.

\section{Constructing $L_i$}

In \cite{AGP}, the authors find a large set of primes $q$ which will eventually divide $p-1$.  In particular, these $q$'s are chosen such that $q-1$ is fairly smooth; hence, when the authors let $L$ be the product of these $q$'s, they are left with an $L$ for which $\lambda(L)$ is small.  Since we are assuming the conjecture, however, we can find $q$'s for which $q-1$ is very smooth; this will allow us to construct an $L$ for which $\lambda(L)$ is even smaller.  As noted above, this construction was previously used in \cite{WrD} and \cite{WrF}.

First, we construct our $L_i$.  As described above, we let
$$J=\prod_{\substack{\frac z2\leq r\leq z, \\ r\mbox{ prime}}}r,$$
where $z$ is a parameter that is large enough for (\ref{conjineq}) to hold for any $l\geq \frac z2$.

We then consider primes of the form $gj+1$ for $g|J$.  Define as before the set
$$\mathcal R_j=\{q\mbox{ prime}:q=gj+1,g|J,\omega(g)=\lfloor \log z\rfloor\}.$$
Note that for any prime in $\mathcal R_j$,
\begin{gather}\label{gbound}
g\leq z^{\log z},
\end{gather}
and hence
$$\left(\log  g\right)^A\leq \left(\log z\right)^{2A}.$$
So we can invoke the conjecture to find that
$$\sum_{j=1}^{\left(\log z\right)^{2A}}|\mathcal R_j|\geq \#\{g|J:\omega(g)=\lfloor \log z\rfloor\},$$
since each choice of $g$ must yield at least one $q$ for $j$ in this range.  Since $j<\frac z2$ and any prime divisor of $g$ is $\geq \frac z2$, we know that $(j,g)=1$ for any $g$.  So any prime $q$ can only appear in at most one set $\mathcal R_j$, and hence the $\mathcal R_j$ are pairwise disjoint.

Now, by the standard combinatorial identity that
\begin{gather}\label{combid}
\left(\begin{array}{c} n\\k\end{array}\right)\geq \left(\frac nk\right)^k,
\end{gather}
we know that
\begin{align*}
\#\{g|J:\omega(g)=\lfloor \log z\rfloor\}\geq &\left(\begin{array}{c}\frac{z}{4\log z} \\ \lfloor \log z\rfloor \end{array}\right)>\left(\frac{z}{5\log^2 z}\right)^{\log z-1}\\
\geq &\left(\frac{z}{5\log^2 z}\right)^{\log z}\left(\frac 1z\right)>\left(\frac{z}{15\log^2 z}\right)^{\log z},
\end{align*}
since $3^{\log z}>z$.
So there must exist a $j_0\leq \left(\log z\right)^{2A}$ such that
$$|\mathcal R_{j_0}|\geq \frac{\left(\frac{z}{15\log^2 z}\right)^{\log z}}{\left(\log z\right)^{2A}}.$$
Choose two disjoint subsets of $R_{j_0}$, each with $\left(\frac{z}{16\log^2 z}\right)^{\log z}$ elements.  We will call these subsets $\mathcal Q_1$ and $\mathcal Q_2$.  We then define
$$L_i=\prod_{q\in \mathcal Q_i}q.$$
For future notational ease, we note that
\begin{gather}\label{zeqn}
\left(\frac{z}{16\log^2 z}\right)^{\log z}=z^{\log z-(2+o(1))\log\log z}.\end{gather}
\section{The sizes of $q$, $L_i$ and $\lambda(L_i)$}
Before we construct the sets $\mathcal P_k^i$, it will be useful to have information about the sizes of $q$, $L_i$, and $\lambda(L_i)$.  First, we find bounds for $q\in \mathcal Q_i$:
\begin{lemma}\label{qbound}
For any $q\in \mathcal Q_i$,
$$\left(\frac z6\right)^{\log z}\leq q\leq 2z^{\log z}\left(\log z\right)^{2A}.$$
\end{lemma}

\begin{proof}
For the upper bound, we use (\ref{gbound}) to find that
\begin{gather*}q=gj_0+1\leq 2gj_0\leq 2z^{\log z}j_0\leq 2z^{\log z}\left(\log z\right)^{2A}.\end{gather*}
For the lower bound, since $g$ has $\lfloor \log z\rfloor$ prime factors and each of the prime factors is $\geq \frac z2$,
\begin{gather*}q\geq \left(\frac z2\right)^{\log z-1}\geq \left(\frac z2\right)^{\log z}\left(\frac 1z\right)\geq \left(\frac z6\right)^{\log z},\end{gather*}
where again we use the fact that $3^{\log z}>z$.
\end{proof}
We use this to bound $L_i$:
\begin{lemma}\label{Lbound}
For $i=1$ or 2,
$$L_i\leq e^{\left(z^{\log z-(2+o(1))\log\log z}\right)\left(\log^2 z+2A\log\log z\right)}.$$
\end{lemma}
\begin{proof}
Using the upper bound for $q$ above as well as the size of $\mathcal Q_i$ given in (\ref{zeqn}), we see that
$$L_i=\prod_{q\in \mathcal Q_i}q\leq \left(2z^{\log z}\left(\log z\right)^{2A}\right)^{z^{\log z-(2+o(1))\log\log z}}=e^{\left(z^{\log z-(2+o(1))\log\log z}\right)\left(\log^2 z+2A\log\log z\right)},$$
where the constant 2 at the front of the penultimate expression is absorbed onto the $o(1)$ term.
\end{proof}
Note that this implies
\begin{gather}\label{logL}\log(L_i)\leq z^{\frac 32\log z}.\end{gather}

By contrast, $\lambda(L)$ is quite a bit smaller:
\begin{lemma}\label{lambdaL}
$$\lambda(L_1L_2)\leq e^{\frac 45z}.$$
\end{lemma}

\begin{proof}
For any prime $q\in \mathcal Q_i$, we know that $q-1|Jj_0$.  Since $$\lambda(L_1L_2)|LCM\left[q-1:q\in \mathcal Q_1\bigcup \mathcal Q_2\right],$$it follows that $\lambda(L_1L_2)|Jj_0$ as well.  We know that the number of primes between $\frac z2$ and $z$ is bounded loosely by $\frac{3z}{4\log z}$ (see e.g. \cite{RS}), and hence
$$\lambda(L_1L_2)\leq Jj_0\leq z^{\frac{3z}{4\log z}}\left(\log z\right)^{2A}\leq z^{\frac{4z}{5\log z}}=e^{\frac 45z}.$$
\end{proof}

\section{The set $\mathcal P_k^1$}
Next, we use $\mathcal Q_1$ and $L_1$ to construct one of the two sets of primes that will yield our Carmichael number.  Define
$$\mathcal P_k=\{p:p=d_1k\nu+1:d_1|L_1,\omega(d_1)=z,(k,\nu L_1L_2)=1\}.$$
We must now determine the size of $\mathcal P_{k}$ for our first choice of $k$:




\begin{lemma}\label{P1bound}
There exists a $k_1\leq 3\nu z^{A}\left(\log z\right)^{2A}$ such that
$$|\mathcal P_{k_1}|\geq z^{z\log z-(2+o(1))z\log\log z}.$$
\end{lemma}
\begin{proof}
Since we require $p=d_1k\nu+1$ and $(k,\nu)=1$, it is sufficient (though not necessary) to consider the congruence
\begin{gather}\label{modforp1}p\equiv 1+d_1\nu\pmod{d_1\nu^2},\end{gather}
since we would then have $$p=d_1\nu(\nu k'+1)+1$$for some $k'$, and hence $k=\nu k'+1$ would be relatively prime to $\nu$.

Note that for any $d_1|L_1$, we can bound the modulus in (\ref{modforp1}) with
\begin{gather}\label{d1bound}d_1\nu^2\leq \nu^2\left(z^{\log z}\left(\log z\right)^{2A}\right)^{z}\leq \nu^2z^{z\log z+2Az\frac{\log\log z}{\log z}}.\end{gather}
Hence,
\begin{gather}\label{k1bound}\left(\log \left(d_1\nu^2\right)\right)^{A}\leq z^{A}\left(\log z\right)^A\left[\log z+3A\log\log z\right]^A<2 z^{A}\left(\log z\right)^{2A}.
\end{gather}
So we see as before that by the conjecture,
\begin{gather}\label{Pk1}\sum_{k'=1}^{2z^{A}\left(\log z\right)^{2A}}|\mathcal P_{\nu k'+1}|\geq \#\{d_1|L_1:\omega(d_1)=z\}.\end{gather}
If $z$ is sufficiently large relative to $\nu$, we have
\begin{gather}\label{qk}
k=\nu k'+1\leq 3\nu z^{A}\left(\log z\right)^{2A}<\left(\frac z6\right)^{\log z}\leq q
\end{gather}
by Lemma \ref{qbound}.  So it follows that $(k,q)=1$ for every $q|L_1L_2$.  Thus, each $p$ appearing on the left-hand side of (\ref{Pk1}) appears exactly once.  Note that
$$\#\{d_1|L_1:\omega(d_1)=z\}\geq \left(\begin{array}{c}z^{\log z-(2+o(1))\log\log z} \\ z\end{array}\right)\geq z^{z\log z-(2+o(1))z\log\log z}.$$
by (\ref{combid}).  So there must exist a $k_1\leq 3\nu z^{A}\left(\log z\right)^{2A}$ such that
$$|\mathcal P_{k_1}|\geq \frac{z^{z\log z-(2+o(1))z\log\log z}}{3\nu z^{A}\left(\log z\right)^{2A}}=z^{z\log z-(2+o(1))z\log\log z}.$$
\end{proof}
\section{The set $\mathcal P_k^2$}

Armed with this definition of $k_1$, we now define another set of primes $\mathcal P_{k_2}$.  The $k_2$ here will be chosen such that for any $p_1\in \mathcal P_{k_1}$ and $p_2\in \mathcal P_{k_2}$, we will have $(p_1-1,p_2-1)=\nu$.  This is what will allow us to prove that $K=\nu$.

\begin{lemma}\label{P2bound}
There exists a $k_2\leq 7\nu^2 z^{2A}\left(\log z\right)^{4A}$ such that
$$|\mathcal P_{k_2}|\geq z^{z\log z-(2+o(1))z\log\log z}$$
and $(k_1,k_2)=1$.
\end{lemma}
\begin{proof}
Again, we choose a congruence condition that will be sufficient though not necessary:
\begin{gather*}
p\equiv 1+d_2\nu \pmod{d_2\nu^2k_1}.
\end{gather*}
In this case, we have
$$p=d_2\nu(\nu k_1k'+1 )+1.$$
Letting $k=\nu k'k_1+1$, we see that $(k,k_1)=1$ and $(k,\nu)=1$.

Taking the log of the bound for $k_1$ in Lemma \ref{P1bound} gives
$$\log k_1\leq 3A\log z.$$
So we can use the bounds in (\ref{d1bound}) and Lemma \ref{P1bound} to find that
\begin{gather}\label{d2bound}d_2\nu^2k_1<\nu^2z^{z\log z+2Az\frac{\log\log z}{\log z}}\left(3\nu z^{A}\left(\log z\right)^{2A}\right)=z^{z\log z+(2A+o(1))z\frac{\log\log z}{\log z}},
\end{gather}
and hence
\begin{gather}\label{k2bound}\left(\log \left(d_2\nu^2k_1\right)\right)^{A}<\left(z\log^2 z+3Az\log \log z\right)^{A}<2z^{A}\left(\log z\right)^{2A}
\end{gather}
when $z$ is sufficiently large.  So as before,
$$\sum_{k'=1}^{2z^{A}\left(\log z\right)^{2A}}|\mathcal P_{\nu k'k_1+1}|\geq \#\{d_2|L_2:\omega(d_2)=z\},$$
From here, the proof is similar to Lemma \ref{P1bound}, beginning with equation (\ref{Pk1}).  We replace the bound for $k$ in (\ref{qk}) with
\begin{align*}
k=\nu k_1k'+1\leq &2\nu z^{A}\left(\log z\right)^{2A}k_1+1\\
\leq &2\nu z^{A}\left(\log z\right)^{2A}\left(3\nu z^{A}\left(\log z\right)^{2A}\right)+1\\
\leq & 7\nu^2 z^{2A}\left(\log z\right)^{4A}.
\end{align*}
Clearly, this is still less than $\left(\frac z6\right)^{\log z}$, and hence the conclusion after (\ref{qk}) still applies.  Thus, there must exist a $k_2\leq 7\nu^2 z^{2A}\left(\log z\right)^{4A}$ such that $(k_1,k_2)=1$ and
$$|\mathcal P_{k_2}|\geq \frac{z^{z\log z-(2+o(1))z\log\log z}}{7\nu^2 z^{2A}\left(\log z\right)^{4A}}=z^{z\log z-(2+o(1))z\log\log z}.$$

\end{proof}
We now prove the claim that was made at the beginning of this section:

\begin{lemma}\label{K=nu}
Let $p_1\in \mathcal P_{k_1}$ and $p_2\in \mathcal P_{k_2}$.  Then $(p_1-1,p_2-1)=\nu$.
\end{lemma}

\begin{proof}
We have shown in Lemmas \ref{P1bound} and \ref{P2bound} that each $k_i$ is coprime to $\nu L_1L_2$ and that $(k_1,k_2)=1$.  Moreover, $(L_1,L_2)=1$, since the two numbers are comprised of nonintersecting sets of prime factors.  So $(L_1k_1\nu,L_2k_2\nu)=\nu$.  Since $\nu|p_1-1|L_1k_1\nu$ and $\nu|p_2-1|L_2k_2\nu$, we then have $(p_1-1,p_2-1)=\nu$.  This proves the lemma.
\end{proof}

\section{Constructing a Carmichael number}
Finally, we construct Carmichael numbers using these sets $\mathcal P_{k_1}$ and $\mathcal P_{k_2}$.  In order to do this, we recall a theorem of van Emde Boas and Kruyswijk \cite{EK} and Meshulam \cite{Me}.  Let $s(L)$ denote smallest number such that a sequence of at least $s(L)$ elements in $(\mathbb Z/L\mathbb Z)^\times$ must contain some nonempty sequence whose product is the identity.  Then we have the following:

\begin{theorem}\label{veb}
For any $L$,
$$s(L)<\lambda(L)(1+\log(\frac{\phi(L)}{\lambda(L)})).$$
Moreover, let $v>t>s(L)$.  Then any sequence of $v$ elements in $(\mathbb Z/L\mathbb Z)^\times$ contains at least $\left(\begin{array}{c} v \\t\end{array}\right)/\left(\begin{array}{c} v \\s(L)\end{array}\right)$ distinct subsequences of length at least $t-s(L)$ and at most $t$ whose product is the identity.
\end{theorem}
In our case, we have the following bound for $s(L_1L_2k_1k_2)$:
\begin{lemma}\label{sbounds}
$$s(L_1L_2k_1k_2)<e^{z}.$$
\end{lemma}
\begin{proof}
First,
$$\lambda(L_1L_2k_1k_2)\leq \lambda(L_1L_2)k_1k_2\leq e^{\frac 45z}\left(21\nu^4z^{3A}\left(\log z\right)^{6A}\right)\leq e^{\frac 56 z}$$
by Lemmas \ref{lambdaL}, \ref{P1bound}, and \ref{P2bound}.  Meanwhile, by (\ref{logL}),
$$\log(L_1L_2k_1k_2)\leq 2\log(L_1L_2)\leq 2z^{3\log z}=2e^{3\log^2 z}<e^{\frac 16 z}$$
when $z$ is large. Thus,
$$s(L_1L_2k_1k_2)<e^{z}.$$
\end{proof}

Now, for $i=1$ or 2, let $F_i(z,X)$ denote the set of integers $n_i\leq X$ such that\\

\par (i) For any $p|n_i$, $p\in \mathcal P_{k_i}$, and

\par (ii) $n_i\equiv 1\pmod{L_1L_2k_1k_2\nu}$.\\

Combining Theorem \ref{veb} and Lemma \ref{sbounds} gives the following:
\begin{lemma}\label{ni}
For $i=1$ or 2,
$$\left|F_i\left(z,z^{z^{z+1}\left(\log z+\left(2A+o(1)\right)\frac{\log\log z}{\log z}\right)}\right)\right|\geq z^{z^{z+1}\left(\log z-(2+o(1))\log\log z\right)}.$$


\end{lemma}
\begin{proof}
We prove this first for $i=2$; the case of $i=1$ can be proven with nearly identical reasoning but slightly better bounds.  To begin, we know that for any $p\in \mathcal P_{k_2}$, $p\equiv 1\pmod{d_2k_2\nu}$ for some $d_2|L_2$.  So it only remains to show that we can combine these $p$ into products $n_1\equiv 1\pmod{L_1L_2k_2}$.

To this end, we recall that
$$|\mathcal P_{k_2}|\geq z^{z\log z-(2+o(1))z\log\log z}$$
by Lemma \ref{P1bound}.  Clearly, this is much bigger than $s(L_1L_2k_1)$, since $s(L_1L_2k_1)\leq s(L_1L_2k_1k_2)$.  So define
\begin{gather*}t=z^z,\\ v=z^{z\log z-(2+o(1))z\log\log z},\end{gather*}
where $v$ is the lower bound for $\mathcal P_{k_2}$ above.

We see that $t<v$.   So by Theorem \ref{veb}, the number of $n_2$ that can be constructed by products of at most $t$ elements and at least $t-s(L)$ elements in $\mathcal P_{k_2}$ is
\begin{align*}  \geq &\left(\begin{array}{c} z^{z\log z-(2+o(1))z\log\log z}\\ z^z\end{array}\right)/\left(\begin{array}{c} z^{z\log z-(2+o(1))z\log\log z}\\ z^{\frac{z}{\log z}}\end{array}\right)\\
\geq & \left(\frac{z^{z\log z-(2+o(1))z\log\log z}}{z^z}\right)^{z^z}/\left(z^{z\log z-(2+o(1))z\log\log z}\right)^{z^{\frac{z}{\log z}}}\\
\geq & \left(z^{z\log z-(2+o(1))z\log\log z}\right)^{z^z-z^{\frac{z}{\log z}}}\\
=& z^{z^{z+1}\log z-(2+o(1))z^{z+1}\log\log z-z^{\frac{z}{\log z}+1}\log z+(2+o(1))z^{\frac{z}{\log z}+1}\log\log z}\\
=& z^{z^{z+1}\left(\log z-(2+o(1))\log\log z\right)}.
\end{align*}
By (\ref{d2bound}) and (\ref{k2bound}), for any $p\in \mathcal P_{k_2}$,
$$p\leq z^{z\log z+\left(2A+o(1)\right)\frac{z\log\log z}{\log z}}.$$
Since any $n_2$ will have at most $t=z^z$ prime factors,
\begin{align*}
n_2&\leq \left(z^{z\log z+\left(2A+o(1)\right)\frac{z\log\log z}{\log z}}\right)^{z^{z}}\\
=&z^{z^{z+1}\left(\log z+\left(2A+o(1)\right)\frac{\log\log z}{\log z}\right)}.
\end{align*}
So
$$\left|F_2\left(z,z^{z^{z+1}\left(\log z+\left(2A+o(1)\right)\frac{\log\log z}{\log z}\right)}\right)\right|\geq z^{z^{z+1}\left(\log z-(2+o(1))\log\log z\right)}.$$
For the case of $i=1$, the proof is the same except that instead of Lemma \ref{P2bound} and equations (\ref{d2bound}) and (\ref{k2bound}), we apply Lemma \ref{P1bound} and equations (\ref{d1bound}) and (\ref{k1bound}).
\end{proof}
Finally, let

$$X=z^{2z^{z+1}\left(\log z+2A\frac{\log\log z}{\log z}\right)}$$
We give the following as a helpful lookup table comparing logs of $X$ to logs of $z$:
\begin{gather*}
\log X=2z^{z+1}\left(\log^2 z+\left(2A+o(1)\right)\log\log z\right),\\
\log\log X=z\log z+O(\log z),\\
\log\log \log X=(1+o(1))\log z,\\
\log\log \log \log X=(1+o(1))\log\log z.
\end{gather*}
We can use Lemma \ref{ni} to prove our main theorem:
\begin{theorem}
$$C_\nu(X)\geq X^{1-(2+o(1))\frac{\log\log \log \log X}{\log \log\log X}}.$$
\end{theorem}
\begin{proof}
From Lemma \ref{ni}, we can construct many $n_1$ and $n_2$ that are 1 modulo $L_1L_2k_1k_2\nu$.  So let $n=n_1n_2$.  Clearly, if $p|n$ then either $p|n_1$, in which case $p-1|\nu L_1k_1$, or else $p|n_2$, in which case $p-1|\nu L_2k_2$.  In either case, $p-1|\nu L_1L_2k_1k_2|n-1$.  So $n$ is a Carmichael number.  Moreover, by Lemma \ref{K=nu}, we know that $K=\nu$ for this choice of $n$.

To find the number of such $n\leq X$, we recall that there are at least
$$z^{z^{z+1}\left(\log z-(2+o(1))\log\log z\right)}$$
choices for $n_1$ with $n_1\leq \sqrt X$, and the same lower bound holds for the number of choices of $n_2$ with $n_2\leq \sqrt X$.  So the number of $n=n_1n_2$ with $n\leq X$ is at least
$$z^{2z^{z+1}\left(\log z-(2+o(1))\log\log z\right)}.$$
This number can be rewritten as
\begin{align*}
z^{2z^{z+1}\left(\log z-(2+o(1))\log\log z\right)}=&z^{2z^{z+1}\left(\log z+2A\frac{\log \log z}{\log z}-(2+o(1))\log\log z\right)}\\
=&Xz^{-2z^{z+1}\left(2+o(1)\log\log z\right)}\\
=& X\left(X^{-(2+o(1))\frac{\log\log z}{\log z+2A\frac{\log\log z}{\log z}}}\right)\\
=& X\left(X^{-(2+o(1))\frac{\log\log z}{\log z}}\right).\\
\end{align*}
Recalling that $\log z=(1+o(1))\log\log\log X$ and $\log\log z=(1+o(1))\log\log\log\log X$, we can write the above as
$$=X\left(X^{-(2+o(1))\frac{\log\log \log \log X}{\log \log\log X}}\right).$$
This proves the theorem.
\end{proof}

\section{Acknowledgements}

We would like to thank Jonathan Webster for asking a question that prompted the writing of this paper.  We also wish to thank an anonymous referee for some very helpful suggestions.  Additionally, we are grateful for a Wofford College Summer Grant that funded this work.

\bibliographystyle{line}

\end{document}